\documentclass[11pt]{article}
\usepackage{graphicx, amssymb, latexsym, float, xcolor, amsmath, tikz}
\usepackage{showlabels, pdfpages, tikz, calc}
\usepackage{lipsum}
\usepackage{tikz-cd}
\usetikzlibrary{arrows,arrows.meta}
\usetikzlibrary{cd,decorations.pathreplacing}
\tikzcdset{arrow style=tikz, diagrams={>=stealth}}
\newcommand\blfootnote[1]{%
  \begingroup
  \renewcommand\thefootnote{}\footnote{#1}%
  \addtocounter{footnote}{-1}%
  \endgroup }

\newtheorem{defin}{}
\newtheorem{saetze}[defin]{}
\newtheorem{conjec}[defin]{}
\newtheorem{lemmas}[defin]{}
\newtheorem{folger}[defin]{}
\newtheorem{bemerk}[defin]{}
\newtheorem{prop}[defin]{}

\newenvironment{theorem}  {\begin{saetze}\it {\bf Theorem:}}{\end{saetze}}

\newenvironment{conjecture}{\begin{conjec}\it {\bf Conjecture:}}{\end{conjec}}

\newenvironment{lemma}    {\begin{lemmas}\it {\bf Lemma:}}{\end{lemmas}}

\newenvironment{remark}   {\begin{bemerk}\it {\bf Remark:}}{\end{bemerk}}
\newenvironment{proof}    {{\it Proof}:}{{\hfill \fillbox \bigskip}}

\newcommand{\fillbox}{\mbox{$\bullet$}}
\newcommand{\ra}{\rightarrow}

\newcommand{\ms}{\mapsto}
\newcommand{\ol}{\overline}

\newcommand{\N}{\mathbb N}

\newcommand{\FF}{\mathcal F}
\newcommand{\TT}{\mathcal T}
\newcommand{\RR}{\mathcal R}
\newcommand{\Z}{\mathbb Z}
\newcommand{\Q}{\mathbb Q}

\newcommand{\G}{\mathcal G}
\newcommand{\B}{\mathcal B}
\renewcommand{\S}{\mathcal S}
\renewcommand{\O}{\mathcal O}
\newcommand{\U}{\mathcal U}

\newcommand{\Hom}{\mathrm{Hom}}

\newcommand{\p}{\mathfrak p}

\newcommand{\hH}{\hat{H}}
\newcommand{\GG}{\mathbb{G}}
\renewcommand{\ll}{\lambda}

\newenvironment{items}{\begin{list}{$\alph{item})$}
{\labelwidth18pt \leftmargin18pt \topsep3pt \itemsep1pt \parsep0pt}}
{\end{list}}

\begin{document}

\title{The frame of the graph associated with \\
the $p$-groups of maximal class}
\author{Bettina Eick, Patali Komma and Subhrajyoti Saha}
\date{\today}
\maketitle

\begin{abstract}
The graph $\G(p)$ associated with the $p$-groups of maximal class is a major 
tool in their classification. We introduce a subgraph of $\G(p)$ called its 
{\em frame}. Its construction is based on the Lazard correspondence.  We 
show that every $p$-group of maximal class has a normal subgroup of order 
at most $p$ whose quotient is in the frame. Since the frame is close to 
the full graph, it offers a new approach towards the classification of these 
groups.
\blfootnote{{\bf Acknowledgement:} The authors thank Charles Leedham-Green 
and Eamonn O'Brien for discussions on this project and comments. The second 
author was supported by an Alexander von Humboldt Foundation Research 
Fellowship.}
\end{abstract}

\section{Introduction}

The classification of $p$-groups of maximal class is a long-standing 
project. It was initiated by Blackburn \cite{Bla58} who classified these 
groups for the small primes $2$ and $3$. The classification for larger
primes was investigated in many publications and is
significantly more complicated than the small prime case; we refer to 
Leedham-Green \& McKay \cite[Chap. 3]{LGM02} for background.
\smallskip

We briefly recall some key aspects of the existing knowledge. For a prime 
$p$ we visualize the $p$-groups of maximal class via their associated graph 
$\G(p)$: the vertices of $\G(p)$ correspond one-to-one to the infinitely 
many isomorphism types of $p$-groups of maximal class and there is an 
edge $G \ra H$ if $H / Z(H) \cong G$ holds. It is known that $\G(p)$ consists 
of an isolated point $C_{p^2}$ and an infinite tree $\TT$ with root $C_p^2$. 
The tree $\TT$ has a unique infinite path $S_2 \ra S_3 \ra \ldots$; this is
called the mainline of $\G(p)$. The 
branch $\B_i$ of $\TT$ is its subtree consisting of all descendants of 
$S_i$ that are not descendants of $S_{i+1}$. Thus each branch $\B_i$ is a 
finite tree with root $S_i$. If $p \leq 3$, then $\B_i$ is a tree of depth 
$1$, but this does not hold for $p \geq 5$. 
\smallskip

The graphs $\G(p)$ for $p \geq 5$ were investigated by Leedham-Green \& McKay 
\cite{LMc76, LMc78a, LMc78, LMc84}. They introduced the {\em constructible 
groups} and investigated the subtrees $\S_i$ of $\B_i$ consisting of them;
these subtrees were later called {\em skeletons}, see for example Dietrich 
\& Eick \cite{DEi17, DEi25}. The skeletons yield significant insights into 
the broad structure of the branches, but many details remain open. More
precisely, \cite[Theorem 11.3.9]{LGM02} shows that almost every finite 
$p$-group $G$ of maximal class has a normal subgroup $N$ of order dividing 
$p^{18(p-1)}$ such that $G/N$ is a skeleton group. This bound is too large 
to provide full structural insights into the branches.
\smallskip

We introduce the {\em frame groups} and investigate the subtrees $\FF_i$
of $\B_i$ consisting of them; such a subtree is called {\em frame} of
$\B_i$. The
frame contains the skeleton and is close 
to the full branch. It thus can be used to understand the details of the
branches. The frame shares many of the nice properties of the skeleton; 
for example, each group in the frame is determined by $(p-3)/2$ parameters. 
In summary, we propose to use the frame to facilitate a full classification 
of $p$-groups of maximal class.

%%%%%%%%%%%%%%%%%%%%%%%%%%%%%%%%%%%%%%%%%%%%%%%%%%%%%%%%%%%%%%%%%%%%%%%%%%%%%
\subsection{Construction of the frame}

Let $K = \Q_p(\theta)$, where $\Q_p$ are the $p$-adic rational numbers and 
$\theta$ is a primitive $p$-th root of unity. Let $\O$ be the maximal order 
of the field $K$ and let $\p$ be the unique maximal ideal in $\O$. For 
each $i \in \N_0$ the power $\p^i$ is the unique ideal of index $p^i$ in $\O$. 
Let $P = \langle \theta \rangle$ cyclic of order $p$ and
\[ \hH_i = \{ \gamma \in \Hom_P( \p^i \wedge \p^i, \p^{2i+1})) \mid 
              \gamma \mbox{ surjective} \}.\]
We consider $\gamma \in \hH_i$ and define the ideal $J_i(\gamma)$ of $\O$ as 
\[ J_i(\gamma) = \langle 
   \gamma( \gamma( x \wedge y ) \wedge z) + 
   \gamma( \gamma( y \wedge z ) \wedge x) + 
   \gamma( \gamma( z \wedge x ) \wedge y) \mid x, y, z \in \p^i \rangle. \]
As $J_i(\gamma)$ is an ideal in $\O$, it follows that $J_i(\gamma) = \p^\ll$ 
for some $\ll \in \N \cup \{\infty\}$, where $\p^\infty$ represents the
trivial ideal. For $m \in \N$ with $i \leq m \leq \ll$, let $L_{i,m}(\gamma)$
be the quotient $\p^i / \p^m$ equipped with the addition $(a+\p^m) + 
(b+\p^m) = (a+b) + \p^m$ and the multiplication
\[ (a+\p^m)(b+\p^m) = \gamma(a \wedge b)+ \p^m.\]
Since $\p^\ll \leq \p^m$, this multiplication satisfies the
Jacobi identity and thus $L_{i,m}(\gamma)$ is a Lie ring. If $L_{i,m}(\gamma)$
has class at most $p-1$, then the Lazard correspondence translates 
$L_{i,m}(\gamma)$ to a finite $p$-group $G_{i,m}(\gamma)$. The cyclic group 
$P$ acts by multiplication with $\theta$ on $L_{i,m}(\gamma)$ and on 
$G_{i,m}(\gamma)$. We define
\[ S_{i,m}(\gamma) = G_{i,m}(\gamma) \rtimes P.\]

In Section \ref{frameconst} we show that each $S_{i,m}(\gamma)$ is a finite 
$p$-group of maximal class. If $m \leq 2i+1$, then $S_{i,m}(\gamma)$ 
corresponds to a vertex on the mainline of $\G(p)$, otherwise 
$S_{i,m}(\gamma)$ corresponds to a vertex in $\B_{i+2}$. The set of 
vertices in $\B_{i+2}$ obtained in this way determines a subtree 
$\FF_{i+2}$ of $\B_{i+2}$ which we call the {\em frame}.
\medskip

Figure \ref{fig9} summarizes the setup. The graph $\G(p)$ consists of an
isolated vertex $C_{p^2}$, the mainline and its branches. The mainline is 
an infinite path that connects the branches $\B_2, \B_3, \ldots$. Each 
branch $\B_i$ contains its frame $\FF_i$ and each frame contains the 
skeleton of $\B_i$. Both, $\B_i$ and $\FF_i$ are trees with 
root $S_i$ of order $p^i$.

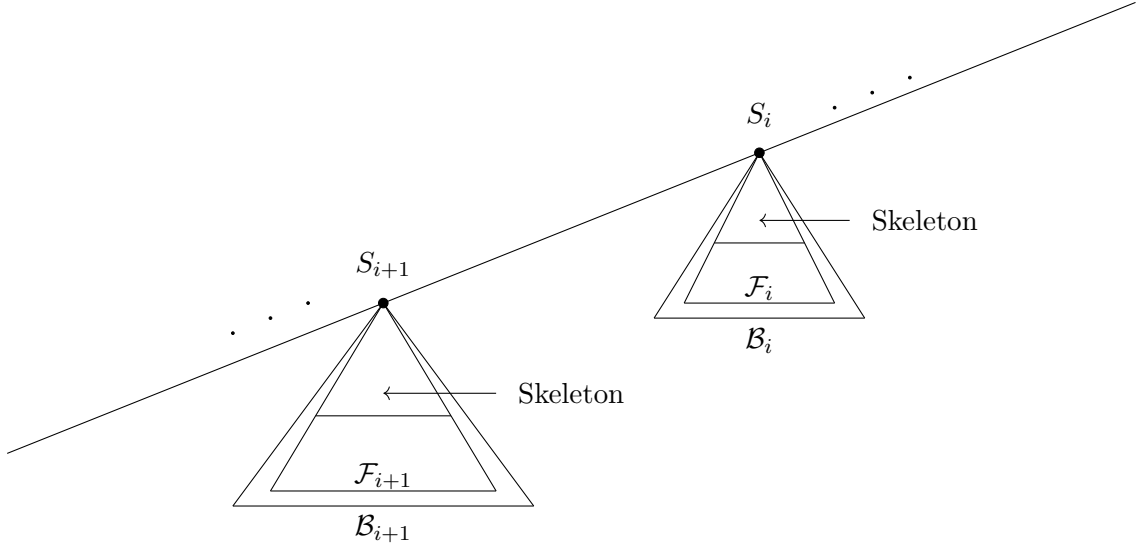
\begin{figure}
\centering
\begin{tikzpicture}
\draw[black] (15,0) -- (0,-6);
\filldraw[black] (10,-2) circle (1.8pt);
\filldraw[black] (5,-4) circle (1.8pt);
\node at (10,-1.5) {$S_i$};
\node at (5,-3.5) {$S_{i+1}$};
\node at (10,-4.5) {$\B_i$};
\node at (5,-7) {$\B_{i+1}$};
\filldraw[black] (12,-1) circle (0.5pt);
\filldraw[black] (11.5,-1.2) circle (0.5pt);
\filldraw[black] (11,-1.4) circle (0.5pt);
\filldraw[black] (4,-4) circle (0.5pt);
\filldraw[black] (3.5,-4.2) circle (0.5pt);
\filldraw[black] (3,-4.4) circle (0.5pt);
\draw[black] (10,-2) -- (9,-4);
\draw[black] (10,-2) -- (11,-4);
\draw[black] (9,-4) -- (11,-4);
\draw[black] (10,-2) -- (8.6,-4.2);
\draw[black] (10,-2) -- (11.4,-4.2);
\draw[black] (8.6,-4.2) -- (11.4,-4.2);
\draw[black] (9.4,-3.2) -- (10.6,-3.2);
\draw[black] (5,-4) -- (3.5,-6.5);
\draw[black] (5,-4) -- (6.5,-6.5);
\draw[black] (3.5,-6.5) -- (6.5,-6.5);
\draw[black] (5,-4) -- (3,-6.7);
\draw[black] (5,-4) -- (7,-6.7);
\draw[black] (3,-6.7) -- (7,-6.7);
\draw[black] (4.1,-5.5) -- (5.9,-5.5);
\draw[->] (6.5,-5.2)-- (5.0,-5.2);
\node at (7.5,-5.2) {$\text{Skeleton}$};
\draw[->] (11.2,-2.9)-- (10.0,-2.9);
\node at (12.2,-2.9) {$\text{Skeleton}$};
\node at (10,-3.8) {$\FF_i$};
\node at (5,-6.3) {$\FF_{i+1}$};
\end{tikzpicture}
\caption{{\small Branch $\B_i$ with root $S_i$ of order $p^i$,
frame $\FF_i$ and its skeleton. Similar for branch $\B_{i+1}$. 
All branches are connected by the mainline of $\G(p)$.}}
\label{fig9}
\end{figure}

%%%%%%%%%%%%%%%%%%%%%%%%%%%%%%%%%%%%%%%%%%%%%%%%%%%%%%%%%%%%%%%%%%%%%%%%%%%%%
\subsection{Main results}
\label{res}

For $i \in \N_0$ and $\gamma \in \hH_i$, let $L_i(\gamma) = L_{i,\ll}(\gamma)$
the maximal Lie ring obtained for $i$ and $\gamma$; this corresponds to 
$\p^i / J_i(\gamma)$. Our first main result investigates the nilpotency 
and class of $L_i(\gamma)$ provided that it is finite. Theorem 
\ref{classbound} implies the following result; it relies on 
Theorem \ref{lcs} and detailed calculations of the degree of commutativity 
for $p$-groups of maximal class.

\begin{theorem}
\label{main0}
Let $p \geq 5$ prime and $\gamma \in \hH_i$ with $J_i(\gamma) \neq \{0\}$. 
\begin{items}
\item[\rm (a)]
If $i > p-2$, then $L_i(\gamma)$ is nilpotent. 
\item[\rm (b)]
If $i > p-1$, then $L_i(\gamma)$ has class at most $p-1$.
\item[\rm (c)]
If $i > 3p - 10$, then $L_i(\gamma)$ has class $3$.
\end{items}
\end{theorem}

The following theorem, proved in Section \ref{frameproof}, shows that the
frame yields significant insight into many details of the branches. Note 
that a non-abelian $p$-group $G$ of maximal class satisfies $|Z(G)| = p$. 

\begin{theorem}
\label{main1}
Let $p \geq 5$ prime and let $i > p+1$. If $G$ is a group
in the branch $\B_i$ of $\G(p)$, then $G/Z(G)$ is in the frame $\FF_i$.
\end{theorem}

Leedham-Green \& McKay \cite{LMc84} proved that each $\gamma \in \hH_i$
can be defined by $(p-3)/2$ parameters. We recall this briefly. 
For $j \in \Z$ let $\sigma_j$ be the Galois automorphism of $K$ mapping 
$\theta \ms \theta^j$. For $a \in \{2, \ldots, (p-1)/2\}$ define 
\[ \vartheta_a : K \wedge K \ra K : 
          x \wedge y \ms \sigma_a(x) \sigma_{1-a}(y) - 
           \sigma_{1-a}(x) \sigma_a(y).\]
Lemma \ref{imgs} asserts that $\vartheta_a( \p^i \wedge \p^i ) = \p^{2i+1}$ 
and thus $\vartheta_a \in \hH_i$ for each $a$ and each $i \in \N_0$. 
Conversely, for each $\gamma \in \hH_i$ there exist coefficients $c_2, 
\ldots, c_{(p-1)/2} \in K$ with $\gamma = \sum_a c_a \vartheta_a$. 
Translated to our setting, each frame group can be defined by $(p-3)/2$ 
parameters. 

Finally, we briefly consider the isomorphism problem for frame groups, see 
Section \ref{frameisom}. 

\begin{theorem}
\label{main2}
Let $p \geq 5$ prime, let $i > p+1$ and let $\gamma, \gamma' \in \hH_i$.
If $S_{i,m}(\gamma) \cong S_{i,m}(\gamma')$, then $L_{i,m}(\gamma) \cong 
L_{i,m}(\gamma')$ via an isomorphism that is compatible with multiplication 
by $\theta$.
\end{theorem}

%%%%%%%%%%%%%%%%%%%%%%%%%%%%%%%%%%%%%%%%%%%%%%%%%%%%%%%%%%%%%%%%%%%%%%%%%%%%%
\section{Preliminaries from number theory}
\label{numbth}

We recall some results from number theory. Many 
of these results are used in the construction of skeleton groups and hence 
are well known. Let $\theta$ denote a $p$-th root of unity over the 
$p$-adic rationals $\Q_p$ and let $K = \Q_p(\theta)$ be the associated 
number field. Let $\O$ be the maximal order in $K$ and $\O = \p^0 > \p^1 
> \ldots$ the unique series of ideals in $\O$ with $[\O : \p^i] = p^i$. Let 
$\kappa = \theta - 1$ and observe that $\p^i$ is generated as an ideal by 
$\kappa^i$. For $j \in \Z$ the map $\sigma_j : \theta \ms \theta^j$ induces 
a Galois automorphism of $K$.

\begin{remark}
Let $P = \langle \theta \rangle$ be the cyclic group of order $p$ and let
$i \in \N_0$. The split extension $\p^i \rtimes P$ is isomorphic to 
the (unique) infinite pro-$p$-group of maximal class.
\end{remark}

Let $i, j \in \N_0$. Then $\p^i$ is an $\O$-module under multiplication.
The wedge product (or exterior square) $\p^i \wedge \p^i$ is an 
$\O$-module under diagonal action and
contains $\p^i \wedge \p^j$ for $j \geq i$. We define 
\[ H = \Hom_P( \O \wedge \O, \O) \;\; \mbox{ and } \;\;
   H_i = \Hom_P( \p^i \wedge \p^i, \O).\]
Now $\gamma \in H$ induces an element in $H_i$ by restriction; we
denote the restricted element also by $\gamma$. The set $\hH_i$ of 
homomorphisms mapping onto $\p^{2i+1}$ is a subset of $H_i$. 
Multiplication by $p$ yields an isomorphism $\p^i \ra \p^{i+(p-1)} : 
x \ms px$. This induces a bijection $\hH_i \ra \hH_{i+(p-1)}$ whose 
inverse is obtained by division with $p$. 

\begin{lemma}
\label{multp}
Let $i \in \N_0$ and $\gamma \in \hH_i$. Then $\gamma \in \hH_j$ for each
$j \in \N_0$ with $i \equiv j \bmod (p-1)$.
\end{lemma}

We show that the homomorphism $\vartheta_a$ defined in Section \ref{res} is 
in $\hH_i$ for each $i \in \N_0$.

\begin{lemma}
\label{imgs}
Let $i, j \in \N_0$, let $a \in \{2, \ldots, (p-1)/2\}$ and let $o_a$ denote 
the order of $a(1-a)^{-1}$ in the multiplicative group $(\Z / p\Z)^*$. Then
\[ \vartheta_a(\p^i \wedge \p^j) = \p^{i+j+\epsilon(i,j)},\]
where $\epsilon(i,j) = 1$ if $o_a \mid (i-j)$ and $\epsilon(i,j) = 0$ 
otherwise. Thus $\vartheta_a(\p^i \wedge \p^i) = \p^{2i+1}$.
\end{lemma}

\begin{proof}
The ideal $\p^i$ has the $\Z_p$-basis $\{ \theta^h \kappa^i \mid 0 \leq h 
\leq p-2\}$. Since $\vartheta_a$ is compatible with $\theta$ and $\Z_p$, it 
suffices to evaluate the terms 
$\vartheta_a( \theta^h \kappa^i \wedge \kappa^j )$
for fixed $i,j$ and every $h \in \{0, \ldots, p-2\}$ to determine the image
$\vartheta_a( \p^i \wedge \p^j)$. Define $s_a = 1 + \theta + \ldots + 
\theta^{a-1}$ and note that $\sigma_a(\kappa) = \kappa s_a$. As $s_a \equiv 
a \bmod \p$, it follows that $s_a \in \U$. Further,
\begin{eqnarray*}
\vartheta_a( \theta^h \kappa^i \wedge \kappa^j ) 
&=& \kappa^{i+j} ( \theta^{ah} s_a^i s_{1-a}^j 
                     - \theta^{(1-a)h} s_{1-a}^i s_a^j ) \\
&=& \kappa^{i+j} u ( \theta^{2ah-h} s_a^{i-j} s_{1-a}^{j-i} - 1)
\end{eqnarray*}
for some unit $u$. Calculating mod $\p$ yields
\[  \theta^{2ah-h} s_a^{i-j} s_{1-a}^{j-i} - 1
            \equiv (a (1-a)^{-1})^{i-j} - 1 \bmod \p. \]
Hence if $o_a \nmid (i-j)$, then $(a(1-a)^{-1})^{i-j} \neq 1 \bmod p$ 
and $\vartheta_a( \theta^h \kappa^i \wedge \kappa^j) = \kappa^{i+j} v$ 
for some unit $v$. This implies that $\vartheta_a( \p^i \wedge \p^j ) 
= \p^{i+j}$. 

In the remainder of the proof we consider the case $o_a \mid (i-j)$. Now
$\vartheta_a(\p^i \wedge \p^j) \leq \p^{i+j+1}$ and it remains to show that 
$\vartheta_a( \p^i \wedge \p^j ) \not \subseteq \p^{i+j+2}$. 

We investigate $\theta^{2ah-h} s_a^{i-j} s_{1-a}^{j-i} \bmod \p^2$. 
As $(1+ e \kappa)^b \equiv 1 + b e \kappa \bmod \p^2$ for $b \in \Z$ and 
$e \in \Q$ it follows that
\[ \theta^{2ah-h} = (1+\kappa)^{2ah-h} \equiv 1 + (2ah-h) \kappa \bmod \p^2. \]
Further,
\[ s_a 
   = \sum_{b=0}^{a-1} \theta^b 
   = \sum_{b=0}^{a-1} (1+\kappa)^b 
   \equiv \sum_{b=0}^{a-1} 1+b \kappa = a + \kappa a(a-1)/2 \bmod \p^2.\]
As $o_a \mid (i-j)$, it follows that $a^{(i-j)} \equiv (1-a)^{(i-j)} \bmod p$. 
Since $p \O = \p^{p-1} \leq \p^2$, this yields that $a^{(i-j)} \equiv 
(1-a)^{(i-j)} \bmod \p^2$. Hence calculating modulo $\p^2$ implies
that
\begin{eqnarray*}
 &&  \theta^{2ah-h} s_a^{i-j} s_{1-a}^{j-i}  \\
 &=& (1+(2ah-h) \kappa) 
          (a+a(a-1)/2 \kappa)^{i-j} ((1-a)+a(a-1)/2 \kappa)^{j-i} \\
 &=& (1+(2ah-h) \kappa) (1+(a-1)/2 \kappa)^{i-j} (1-a/2 \kappa)^{j-i} \\
 &=& 1 + \kappa ( h(2a-1) + (i-j) (a-1)/2 - (j-i) a/2) \\
 &=& 1 + \kappa ( (2a-1)(h + (i-j)/2)). 
\end{eqnarray*}
Choose $h \in \{ 0, \ldots, p-2\}$ with $h + (i-j)/2 \neq 0 \bmod p$.
As $2a-1 \neq 0 \bmod p$, it follows that 
$\theta^{2ah-h} s_a^{i-j} s_{1-a}^{j-i} \neq 1 \bmod \p^2$. This
yields the desired result.
\end{proof}

While the images on $p^i \wedge \p^j$ can be determined explicitly for
$\vartheta_a$, this is not so easy for arbitrary homomorphisms $\gamma 
\in H$. We recall the following bound from \cite[Lemma 3.2]{LMc76} and
include its elementary proof for completeness.

\begin{lemma}
\label{gammimgs}
Let $i, j \in \N_0$ and $\gamma \in H$. Then
$\gamma( \p^i \wedge \p^j ) \leq \p^{i+j-(p-2)}$.
\end{lemma}

\begin{proof}
Let $i \in \N_0$ and consider $j = 0$. Write $i = s(p-1)+r$ with $r \in
\{0, \ldots, p-2\}$. Then $\p^i = p^s \p^r$ and $\gamma(\p^i \wedge \O)
= p^s \gamma( \p^r \wedge \O)$. As $\gamma( \p^r \wedge \O) \leq \O
\leq \p^{r-(p-2)}$, it follows that $\gamma(\p^i \wedge \O) \leq p^s
\p^{r-(p-2)} = \p^{s(p-1)+r-(p-2)} = \p^{i-(p-2)}$. The result
follows for all $i$ and $j = 0$.\\
Now suppose that the result is proved for all $i$ and some fixed $j
\geq 0$. We show that it holds for $j+1$. Let $a \in \p^i$
and $b \in \p^{j+1}$. Then
\begin{eqnarray*}
&& \gamma( a \wedge b) \\
 &=& \gamma( a \wedge b \kappa^{-1} (\theta-1)) \\
 &=& \gamma( a \wedge b \kappa^{-1} \theta)
                - \gamma( a \wedge b \kappa^{-1}) \\
 &=& \theta \gamma( \theta^{-1} a \wedge b \kappa^{-1})
                - \gamma( a \wedge b \kappa^{-1}) \\
 &=& (\kappa+1) \gamma( \theta^{-1} a \wedge b \kappa^{-1})
                - \gamma( a \wedge b \kappa^{-1}) \\
 &=& \kappa \gamma( \theta^{-1} a \wedge b \kappa^{-1})
     +   \gamma( \theta^{-1} a \wedge b \kappa^{-1})
                - \gamma( a \wedge b \kappa^{-1}) \\
 &=& \kappa \gamma( \theta^{-1} a \wedge b \kappa^{-1})
     -   \gamma( (a - \theta^{-1} a) \wedge b \kappa^{-1}).
\end{eqnarray*}
Now $b \kappa^{-1} \in \p^j$, so induction applies. Further,
$\theta^{-1} a \in \p^i$ and $a - \theta^{-1} a = \theta^{-1} a \kappa
\in \p^{i+1}$. Hence both summands are in $\p^{i+j+1-(p-2)}$ by induction.
\end{proof}

Next, we recall some well-known results on the splitting of
homomorphisms in $H$ into a linear combination of the elements 
$\vartheta_a$. Let $l = (p-3)/2$ and for $2 \leq a \leq l+1$ define
\[ u_a = (\theta^a-1)(\theta^{1-a}-1) \in \p^2.\]
Let $V_i$ be the diagonal matrix with diagonal entries 
$(\theta^a - \theta^{1-a}) u_a^i / \kappa^{2i+1}$ for $2 \leq a \leq l+1$ 
and let $B$ be the Vandermonde matrix with entries $u_a^{j-1}$ for $2 
\leq a \leq l+1$ and $1 \leq j \leq l$. 

\begin{theorem}
\label{span}
Let $i \in \N_0$ and $l = (p-3)/2$. 
\begin{items}
\item[\rm (a)]
For each $\gamma \in H$ there exist (unique) $c_2, \ldots, 
c_{l+1} \in K$ with $\gamma = \sum_a c_a \vartheta_a$. 
\item[\rm (b)]
If $\gamma = \sum c_a \vartheta_a \in \hH_i$, then 
$(c_2, \ldots, c_{l+1}) V_i B \in \O^l \setminus \p^l$. 
\end{items}
\end{theorem}

\begin{proof}
(a) 
This follows from \cite[Theorem 8.3.1]{LGM02}.\\
(b) 
Each homomorphism $\gamma$ is defined by its images on the elements
$\kappa^{i+j} \wedge \kappa^{i+j-1}$ for $1 \leq j \leq l$, see 
\cite[Prop. 8.3.5]{LGM02}. A direct calculation yields that
\[ \vartheta_a( \kappa^{i+j} \wedge \kappa^{i+j-1} ) 
   = \kappa^{2i+1} v_a u_a^{j-1}\]
for $2 \leq a \leq l+1$ and $1 \leq j \leq l$. 
As $\gamma = \sum c_a \vartheta_a$, it follows that 
$\gamma( \kappa^{i+j} \wedge \kappa^{i+j-1})$ corresponds to the 
$j$-th entry in $\kappa^{2i+1} (c_2, \ldots, c_{l+1}) V_i B$. Thus 
$\gamma$ maps surjectively onto $\p^{2i+1}$ if and only if $(c_2, \ldots, 
c_{l+1})V_i B \in \O^l \setminus \p^l$. 
\end{proof}

%%%%%%%%%%%%%%%%%%%%%%%%%%%%%%%%%%%%%%%%%%%%%%%%%%%%%%%%%%%%%%%%%%%%%%%%%%%%%
\section{Construction of groups of maximal class}
\label{frameconst}

Recall that a Lie ring is an additive group with a bracket-multiplication
that is anticommutative and satisfies the Jacobi identity. A Lie $p$-ring
is nilpotent of $p$-power size. In this section we use the Lazard 
correspondence and Lie ring construction of the introduction to determine
$p$-groups of maximal class. Jaikin-Zapirain \& Vera-Lopez \cite{JVL00}
use the Lazard correspondence in an alternative way to investigate 
$p$-groups of maximal class.

%%%%%%%%%%%%%%%%%%%%%%%%%%%%%%%%%%%%%%%%%%%%%%%%%%%%%%%%%%%%%%%%%%%%%%%%%%%%%
\subsection{Lie $p$-rings}

Let $i \in \N_0$ and let $\gamma \in \hH_i$. For $x, y \in \p^i$ define
\[ [x,y] := \gamma(x \wedge y) \mbox{ for } x, y \in \p^i. \]
Then $[.,.]$ is anticommutative. We recall from the introduction that
\[ J_i(\gamma) = \langle [[x,y],z] + [[y,z],x] + [[z,x],y] \mid 
x,y,z \in \p^i \rangle.\]
Then $J_i(\gamma) = \p^\ll$ with $\ll \in \N \cup \{ \infty \}$ and 
$L_i(\gamma) = \p^i / \p^\ll$ is the induced Lie ring. For $m \in \N$ 
with $i \leq m \leq \ll$ we obtain the Lie ring $L_{i,m}(\gamma)$ 
corresponding to $\p^i / \p^m$.  Note that $L_{i,m'}(\gamma)$ is a 
quotient of $L_{i,m}(\gamma)$ for $m' \leq m$ and $L_{i,m}(\gamma)$ is
a quotient of $L_i(\gamma)$.

The next theorem has an elementary proof and is a first step towards 
understanding the size and nilpotency class of $L_{i,m}(\gamma)$. 
Recall that the lower central series of a Lie ring $L$ is defined by 
$L^{(1)} = L$ and $L^{(j)} = [L^{(j-1)}, L]$ for $j > 1$. 

\begin{theorem}
\label{lcs}
Let $i > p-2$, let $\gamma \in \hH_i$ and let $J_i(\gamma) = \p^\ll$.
\begin{items}
\item[\rm (a)]
$\ll \geq 3i+3-p$. 
\item[\rm (b)]
If $J_i(\gamma) \neq \{0\}$, then $L_i(\gamma)$ is nilpotent and has 
size at least $p^{2i+3-p}$.
\item[\rm (c)]
If $J_i(\gamma) \neq \{0\}$ and $i$ is sufficiently large, then 
$L_i(\gamma)$ has class $3$.
\end{items}
\end{theorem}

\begin{proof}
Write $L = L_i(\gamma)$. Define $w_1 = i$ and 
$\p^{w_{k+1}} = \gamma( \p^{w_k} \wedge \p^i)$ for $k \geq 1$. Then 
$w_2 = 2i+1$ by the choice of $\gamma$ and $w_{k+1} \geq w_k + i - (p-2)$
for $k \geq 2$ by Lemma \ref{gammimgs}. The lower central series of
$L$ has the terms $L^{(k)} = \p^{w_k}/\p^\ll$ as long as $w_k \leq \ll$, 
since the multiplication in $L$ is defined via $\gamma$. \\
(a) 
By construction, $\p^\ll \leq \p^{w_3}$ and thus $\ll \geq w_3 \geq 
w_2 + i - (p-2) = 3i+3-p$. \\
(b) 
If $i > p-2$, then $w_{k+1} > w_k$ for $k \geq 2$ by the first part of this
proof. Hence if $L_i(\gamma)$ is finite, then it is nilpotent. \\
(c) 
Write $d = p-1$ and let $r \in \{0, \ldots d-1\}$ with $r \equiv i \bmod d$. 
We consider the different Lie 
rings arising from $\gamma \in \hH_j$ for $j \in r + d \N_0$ as in Lemma
\ref{multp}. Write $J_j(\gamma) = \p^{\ll(j)}$ and $L(j) = L_j(\gamma)$.
Define $w_k(j)$ via $\p^{w_k(j)} = \gamma( \p^{w_{k-1}(j)} \wedge \p^j)$ 
corresponding to the $k$-th term of the lower central series of $L(j)$. 
Since $\gamma$ is compatible with multiplication by $p$, 
\[ w_k( r+dh ) = w_k(r) + kdh \mbox{ and } \ll(r+dh) = \ll(r)+3dh \]
for each $h \in \N_0$. Thus $w_3(j)$ grows by $3d$ with increasing $j$.
Similarly, if $\ll(j)$ is finite, then it grows by $3d$. However, $w_4(j)$ 
grows by $4d$. Hence if $\ll(r)$ is finite and $h \geq (\ll(r)-w_4(r))/d$, 
then $\p^{w_4(r+dh)} \leq \p^{\ll(r+dh)}$ and $L_i(\gamma)$ has class $3$ 
for $i = r+hd$.
\end{proof}

%%%%%%%%%%%%%%%%%%%%%%%%%%%%%%%%%%%%%%%%%%%%%%%%%%%%%%%%%%%%%%%%%%%%%%%%%%%%%
\subsection{The Lazard correspondence}
\label{lazard}

Let $i \in \N_0$ and $\gamma \in \hH_i$. Let $J_i(\gamma) = \p^\ll$ and 
$m \in \{i, \ldots, \ll\}$. Thus $L = L_{i,m}(\gamma)$ has class at most
$p-1$. The Lazard correspondence applies and yields a group $\GG(L)$ of 
order $p^{m-i}$. As a set, $\GG(L)$ has the same elements as $L$: the elements 
of $\p^i / \p^m$. The group operations on $\GG(L)$ can be expressed as 
formulae in the operations of $L$, and their lower central series coincide.
The most commonly occurring case is that $L$ has class 3; we outline
the corresponding group operation.

\begin{lemma}
Let $G = \GG(L)$ and assume that $L$ has class at most $3$. 
\begin{items}
\item[$\bullet$]
The multiplication in $G$ (on the left) translates to the following
(with the Lie bracket on the right):
\[ (a+\p^m)(b+\p^m) = (a+b+ \frac{1}{2}[a,b]
   + \frac{1}{12} ([a,[a,b]] + [b,[b,a]])) + \p^m,\]
\item[$\bullet$]
The commutator in $G$ (on the left) translates to the following 
(with the Lie bracket on the right):
\[ [a+\p^m, b+\p^m] = ([a,b] + \frac{1}{2} ([a,[a,b]]+[b,[b,a]])) + \p^m.\]
\item[$\bullet$]
A power in $G$ translates to multiplication in $L$:
\[ (a+\p^m)^x = xa + \p^m \mbox{ for } x \in \Z.\]
\end{items}
\end{lemma}

As the Lie bracket in $L$ is compatible with multiplication in $\theta$, and 
the multiplication in $\GG(L)$ is based on this, the multiplication 
in $\GG(L)$ is also compatible with the multiplication by $\theta$. Thus we 
have proved the following.

\begin{lemma}
\label{multheta}
Multiplication with $\theta$ defines a ring homomorphism on $L$ and
a group homomorphism on $\GG(L)$.
\end{lemma}

As before, let $P = \langle \theta \rangle$ and recall $S_{i, m}(\gamma) 
= \GG(L_{i,m}(\gamma)) \rtimes P$. 

\begin{theorem}
\label{grpconst}
Let $i \in \N_0$ and $\gamma \in \hH_i$. Then $S = S_{i,m}(\gamma)$ is a 
group of order $p^{m-i+1}$ and maximal class. If $m \leq 2i+1$, then $S$
is a group on the mainline of $\G(p)$, otherwise it is in the branch
$\B_{i+2}$.
\end{theorem}

\begin{proof}
Write $L = L_{i,m}(\gamma)$ and $S = S_{i,m}(\gamma)$. Consider $G = 
\GG(L_{i,m}(\gamma))$ as a subgroup of $S$. Recall that $\gamma( \p^i 
\wedge \p^i ) = \p^{2i+1}$ and this corresponds to $L'$. Thus if 
$m \leq 2i+1$, then $L'$ is trivial. Otherwise, $L'$ corresponds to
$\p^{2i+1}/\p^m$. The Lazard correspondence translates this to $G$.
As $S/G'$ is the largest mainline quotient of $S$, the result follows.
\end{proof}

%%%%%%%%%%%%%%%%%%%%%%%%%%%%%%%%%%%%%%%%%%%%%%%%%%%%%%%%%%%%%%%%%%%%%%%%%%%%%
\subsection{The class of a Lie ring}
\label{lieclass}

The proof of the following theorem is based on Theorem \ref{lcs} and the 
theory of the degree of commutativity of finite $p$-groups. Theorem 
\ref{main0} follows directly from it.

\begin{theorem}
\label{classbound}
Let $i \in \N_0$ and $\gamma \in \hH_i$ with $J_i(\gamma) \neq \{0\}$. 
Then $L_i(\gamma)$ has class at most 
\[ 3+\frac{2p-8}{i-(p-2)}.\]
\end{theorem}

\begin{proof}
We retain the notation of Theorem \ref{lcs} and its proof. Let $d=p-1$ and 
$r \in \{0, \ldots, d-1\}$ with $r \equiv i \bmod d$. Choose $j \in r + d \Z$ 
large enough so that $L_j(\gamma)$ has class $3$; such $j$ exists by Theorem 
\ref{lcs}. Let $J_j(\gamma) = \p^\ll$ and let $S = S_{j, \ll}(\gamma)$
be the associated group of maximal class as in Theorem \ref{grpconst}.
Now \cite[Cor. 3.4.12]{LGM02} implies that $|\gamma_3(\GG(L_j(\gamma)))|
\leq p^{2p-8}$, where $\gamma_3(G)$ is the third term of the lower
central series of $G$. (Note that the subgroup $P_1$ used in 
\cite[Cor. 3.4.12]{LGM02} is defined in \cite[p. 56]{LGM02} as a
two-step centralizer of $S$ and coincides with $\GG(L_j(\gamma))$.) 

By the Lazard correspondence, $|L_j(\gamma)^{(3)}| \leq
p^{2p-8}$, where $L^{(3)}$ is the third term of the lower central series
of the Lie ring $L$. As in the proof of Theorem \ref{lcs}, define $w_k(j)$
via $w_1(j) = j$ and $w_2(j) = 2j+1$ and $\p^{w_{k+1}(j)} = \gamma( 
\p^{w_k(j)} \wedge \p^j)$. This yields a series of ideals $\p^{w_k(j)}$
in $\O$. Since $\O$ has a unique chain of ideals, there
exists $c_j \in \N$ so that
\[ p^{w_1(j)} \geq \p^{w_2(j)} \geq \p^{w_3(j)} \geq \ldots 
  \p^{w_{c_j}(j)} \geq \p^{\ll(j)} \geq \p^{w_{{c_j}+1}(j)} \geq \ldots.\]
The lower central series of $L_j(\gamma)$ corresponds to the quotients 
of this series; more precisely, $L_j(\gamma)^{(k)}$ corresponds to 
$\p^{w_k(j)}/\p^{\ll}$ for $1 \leq k \leq c_j$ and $c_j$ is the class
of $L_j(\gamma)$. Thus $(\ll(j) - w_3(j)) 
\leq 2p-8$ follows. Theorem \ref{lcs} and its proof yield that 
$\ll(r+hd) - w_3(r+hd) = \ll(j)-w_3(j)$ for all $h \in \N_0$. Hence 
$\ll(i)-w_3(i) = \ll(j) - w_3(j) \leq 2p-8$.

Lemma \ref{gammimgs} implies that
$w_{k-1}(i)-w_k(i) \leq i-(p-2)$ for $k \in \N$. Thus
\[ \ll(i) - w_3(i) 
   \geq w_{c_i}(i) - w_3(i) 
   = \sum_{k=3}^{c_i-1} w_{k+1}(i) - w_k(i)
   \geq (c_i-3) (i-(p-2)).\]
In summary, $2p-8 \geq \ll(i)-w_3(i) \geq (c_i-3)(i-(p-2))$ follows.
Equivalently, $c_i \leq 3+(2p-8)/(i-(p-2))$. As $c_i$ coincides with the class
of $L_i(\gamma)$, the result follows.
\end{proof}

%%%%%%%%%%%%%%%%%%%%%%%%%%%%%%%%%%%%%%%%%%%%%%%%%%%%%%%%%%%%%%%%%%%%%%%%%%%%%
\section{The frame of a branch}
\label{frameproof}

The {\em frame} $\FF_{i+2}$ of a branch $\B_{i+2}$ is its subtree 
consisting of all groups $S_{i,m}(\gamma)$ with $\gamma \in \hH_i$
as constructed in Section \ref{lazard}. As $S_{i,m-1}(\gamma)$ is 
a quotient of $S_{i,m}(\gamma)$, it follows that $\FF_{i+2}$ is a 
full subtree of $\B_{i+2}$. The
skeleton groups (or constructible groups) are a special case of
frame groups obtained by choosing $m$ so that $L_{i,m}(\gamma)$ has
class $2$. Hence the skeleton of $\B_{i+2}$ is a subtree of 
$\FF_{i+2}$.
\smallskip

In this section we prove Theorem \ref{main1}.
We consider $p \geq 5$ prime and $i > p$.
Let $G$ be a group in the frame $\B_{i+2}$ of $\G(p)$. 
Then $|G| = p^n$ with $n \geq i+2 > p+2$ and thus $G$ is non-abelian 
and $Z(G)$ has order $p$. Further, $G$ and $G/Z(G)$ both have positive 
degree of commutativity by \cite[Theorem 3.3.5]{LGM02}. We
show that $G/Z(G)$ is a frame group; that is, there exists $\gamma \in 
\hH_i$ with $G/Z(G) \cong S_{i,m}(\gamma)$, where $m = n +i - 1$.

Let $G = \gamma_1(G) > \gamma_2(G) > \ldots > \gamma_c(G) > \gamma_{c+1}(G) 
= \{1\}$ denote the lower central series of $G$ and let $M = 
C_G(\gamma_2(G)/ \gamma_4(G))$ be the associated two-step centralizer. 
(Then $M$ coincides with $P_1$ and $K_2$ in the notation of 
\cite[Chap.\ 3]{LGM02}.) Let 
$s \in G \setminus M$, let $s_1 \in M \setminus \gamma_2(G)$ and define 
$s_j = [s_{j-1}, s]$ for $j > 1$. (This notation corresponds to that
in \cite[Sec. 3.2]{LGM02}.)

\begin{lemma}
\label{split}
$G/Z(G)$ is a split extension of $M/Z(G)$ by a cyclic group of order $p$.
\end{lemma}

\begin{proof}
The power $s^p$ is central in $G$ by \cite[Lemma 3.3.7]{LGM02}. Hence 
$s^p \equiv 1 \bmod Z(G)$ and $G/Z(G) \cong M/Z(G) \rtimes 
\langle s \rangle$.
\end{proof}

Shepherd \cite{She70} shows that $cl(M) \leq (p+1)/2 < p$. Thus 
the Lazard correspondence applies to $M$ and yields a Lie $p$-ring
$L(M)$. Similarly, $M/Z(G)$ corresponds to a Lie $p$-ring $L(M/Z(G))$
and this is a quotient of $L(M)$.

\begin{lemma}
\label{addi}
The additive group of $L(M)$ is isomorphic to the additive group of
$\p^i / \p^m$ and conjugation by $s$ induces an automorphism on $M$
that translates to multiplication by $\theta$ on $\p^i / \p^m$.
\end{lemma}

\begin{proof}
By construction, $M = \langle s_1, \ldots, s_{n-1} \rangle$.
As sets, $M$ and $L(M)$ can be identified. The Lazard correspondence 
maps $g^x$ in $M$ to $x g$ in $L(M)$ for $x \in \Z$ and vice versa. 
Hence the power structure of $M$ translates to the additive structure 
of $L(M)$. By \cite[Prop.\ 3.3.8]{LGM02} the additive group of $L(M)$ 
has rank $p-1$ and is almost homocyclic. 
By construction, the conjugation by $s$ on $M$ has the form $s_j^s = 
s_j s_{j+1}$ for $j \geq 1$. This coincides with the multiplication by
$\theta$ on the generators $\{ \kappa^j \mid i \leq j \leq i+(p-1) \}$ 
of $\p^i$. 
\end{proof}

The next lemma completes the proof of Theorem \ref{main1}.

\begin{lemma}
\label{multi}
There exists $\gamma \in \Hom_P(\p^i \wedge \p^i, \p^{2i+1})$ with
$L_{i,m-1}(\gamma) \cong L(M/Z(G))$.
\end{lemma}

\begin{proof}
Since $G$ is in $\B_{i+2}$, its largest mainline quotient
$G/U$ has order $p^{i+2}$ and satisfies $G/U \cong S_{i+2}$. The group 
$S_{i+2}$ is isomorphic to $(\p^i / \p^{2i+1}) \rtimes P$ and thus has
an abelian two-step centralizer. Hence $M' \leq U$ for the two-step 
centralizer $M$ of $G$.

If $M'$ is trivial, then $G/Z(G) \cong (\p^i / \p^m) \rtimes P$ by
Lemmas \ref{split} and \ref{addi}. Thus $G/Z(G)$ is a mainline group 
and $L(M/Z(G)) \cong L_{i,m-1}(\gamma)$, where $\gamma$ is the 
trivial homomorphism.

We now assume that $M'$ is non-trivial. Then
$Z(G) \leq M'$ and $G/M'$ is a mainline group by Lemmas \ref{split}
and \ref{addi}. Thus $U = M'$. By Lemma \ref{addi}, we identify 
$L(M)$ with $\p^i/\p^m$ as additive group. As $L(M)'$ corresponds to $M'$ 
under the Lazard correspondence, $[L(M) : L(M)'] = p^{i+1}$. 
Hence the multiplication in the Lie $p$-ring $L(M)$ is a bilinear 
antisymmetric map of the form 
$\gamma : \p^i/\p^m \wedge \p^i/\p^m \ra \p^{2i+1}/\p^m$.
As the multiplication in $M$ is compatible with conjugation by $s$, 
the multiplication in $L(M)$ is compatible with the 
multiplication by $\theta$, see Lemma \ref{addi}. Hence 
$\gamma \in \Hom_P(\p^i/\p^m \wedge \p^i/\p^m, \p^{2i+1}/\p^m)$.
We note that $m > 2i+1$, otherwise $M'$ is trivial.
Write $H = \Hom_P(\p^i \wedge \p^i, \p^{2i+1}/\p^m)$. Note that
there are two natural homomorphisms
\begin{eqnarray*}
\psi &:& H \ra \Hom_P(\p^i/\p^m \wedge \p^i/\p^m, \p^{2i+1}/\p^m), 
   \mbox{ and } \\
\phi &:& \Hom_P(\p^i \wedge \p^i, \p^{2i+1}) \ra H.
\end{eqnarray*}
The homomorphism $\psi$ is surjective and thus $\gamma$ is in its image.
We replace $\gamma$ by a preimage under $\psi$, which we also call $\gamma$.
Thus $\gamma \in H$ now. Leedham-Green \& McKay \cite[Theorem 8.3.7]{LGM02}
show that the image $I$ of $\phi$ has index $p$ in $H$ and is supplemented 
by $N = \Hom_P(\p^i \wedge \p^i, \p^{m-1}/\p^m)$. Thus $\gamma = \gamma_I 
+ \gamma_N$ with $\gamma_I \in I$ and $\gamma_N \in N$. Let $\ol{\gamma}$
denote a preimage of $\gamma_I$ under $\phi$. 

Finally, $Z(G)$ corresponds to $Z(L(G))$ and thus to $\p^{m-1}/\p^m$. As 
$\gamma_N$ vanishes modulo $\p^{m-1}/\p^m$ by construction, 
$L(M/Z(G)) = L_{i,m-1}(\ol{\gamma})$ which yields the desired result.
\end{proof}

%%%%%%%%%%%%%%%%%%%%%%%%%%%%%%%%%%%%%%%%%%%%%%%%%%%%%%%%%%%%%%%%%%%%%%%%%%%%%
\section{The isomorphism problem for frame groups}
\label{frameisom}

Let $G$ and $H$ be groups in the frame of $\B_{i+2}$. Then there exist 
$\gamma, \gamma' \in \hH_i$ with $G = S_{i,m}(\gamma)$ and 
$H = S_{i,m}(\gamma')$. The two-step centralizer $M_G$ of
$G$ coincides with $\GG(L_{i,m}(\gamma))$ and similarly, $M_H = 
\GG(L_{i,m}(\gamma')$. The Lazard correspondence implies the following.

\begin{theorem}
\label{isom}
If $G \cong H$, then $L(M_G) \cong L(M_H)$ via an isomorphism that
is compatible with the multiplication by $\theta$.
\end{theorem}

\begin{proof}
If $G \cong H$, then $M_G \cong M_H$, since the two-step centralizers
are fully invariant in their respective parent groups by construction. Thus
the Lazard correspondence implies $L(M_G) \cong L(M_H)$ and this
isomorphism is compatible with the multiplication by $\theta$.
\end{proof}

Recall that $\U$ is the unit group of $\O$ and let $\rho_a(u) 
= u^{-1} \sigma_a(u) \sigma_{1-a}(u)$ for $u \in \U$. 

\begin{theorem} 
\label{conjZ}
Let $\gamma = \sum c_a \vartheta_a$ and $\gamma' = \sum c_a' \vartheta_a$
both be in $\hH_i$. If there exist $u \in \U$ and $\sigma \in Gal(K)$ with
\[ \sigma(c_a') \equiv \rho_a(u) c_a \bmod \p^m
   \;\; \mbox{ for } \;\; 2 \leq a \leq (p-1)/2,\]
then $S_{i,m}(\gamma) \cong S_{i,m}(\gamma')$.
\end{theorem}

\begin{proof}
Let $u \in \U$ and consider $\gamma = \sum c_a \vartheta_a$ 
and $\gamma' = \sum c_a' \vartheta_a$ with $c_a' = \rho_a(u) c_a$. Then the
map $\p^i / \p^m \ra \p^i / \p^m : a + \p^m \ms u a + \p^m$ induces an 
isomorphism $L_i(\gamma') \ra L_i(\gamma)$, since 
$\vartheta_a( ux \wedge uy) = \rho_a(u) \vartheta_a( x \wedge y)$ holds.

Similarly, let $\sigma \in Gal(K)$ and  consider the map 
$\p^i / \p^m \ra \p^i / \p^m : a + \p^m \ms \sigma(a) + \p^m$. 
As $\sigma$ is compatible with each $\vartheta_a$, this induces
an isomorphism $L_i(\gamma') \ra L_i(\gamma)$, where $\gamma' =
\sum \sigma(c_a) \vartheta_a$.

This yields the desired result.
\end{proof}

If $M_G$ and $M_H$ have class $2$, then the converse of Theorem
\ref{conjZ} follows via the solution of the isomorphism problem for 
skeleton groups, see \cite{LMc84}. For class at least $3$, the 
converse of Theorem \ref{conjZ} remains open.

%%%%%%%%%%%%%%%%%%%%%%%%%%%%%%%%%%%%%%%%%%%%%%%%%%%%%%%%%%%%%%%%%%%%%%%%%%%%%
\section{Conjectures}

We investigated the frame of $\G(5)$ in \cite{EKS25} and explored the 
frames of $\G(7), \G(11)$ and $\G(13)$ computationally. Based 
on this, we propose the following conjectures.

\begin{conjecture}
Let $p \geq 5$ prime, let $i \in \N_0$ and $\gamma \in \hH_i$. Then
$J_i(\gamma) \neq \{0\}$.
\end{conjecture}

This conjecture implies that $L_i(\gamma)$ is always finite, and hence
nilpotent by Theorem \ref{lcs}. Further, Theorem \ref{classbound} yields
that it has class at most $p-1$ if $i > p-1$. 

\begin{conjecture} 
\label{conjX}
Let $p \geq 5$ prime and $i > p+1$. Then the leaves of the frame $\FF_i$ 
are terminal groups in $\B_i$.
\end{conjecture}

Natural questions arise. How does the sequence of frames $\FF_i,
\FF_{i+1}, \ldots$ grow with $i$? What is the structure of the 
branches $\B_i$ outside the frames $\FF_i$? We define 
the {\em twig} $\RR(G)$ for a group $G$ in $\FF_i$ as the subtree of 
$\B_i$ consisting of all descendants of $G$ that are not in the frame 
$\FF_i$. By construction, $\RR(G)$ is the tree with root $G$ and Theorem 
\ref{main1} asserts that it has depth at most $1$. The following is a 
variation of Conjecture W as proposed by 
Eick, Leedham-Green, Newman \& O'Brien \cite{ELNO13}.

\begin{conjecture}
\label{periodII}
Let $p \geq 5$ prime. Then there exists $e = e(p)$ and $f = f(p)$ with
$(p-1) \mid f$ so that for each $i \geq e$ and each $\gamma \in \hH_i$
\[ \RR( S_{i+f, m}(\gamma) ) \cong \RR(S_{i,m}(\gamma)). \]
\end{conjecture}

%%%%%%%%%%%%%%%%%%%%%%%%%%%%%%%%%%%%%%%%%%%%%%%%%%%%%%%%%%%%%%%%%%%%%%%%%%%%%
\bibliographystyle{abbrv}

\def\cprime{$'$} \def\cprime{$'$}

\end{document}